\pdfoutput=1

\documentclass[12pt]{amsart}

\usepackage{color}
\usepackage{graphicx}
\usepackage[mathcal, mathscr]{euscript}
\usepackage{bbm}
\usepackage{float}
\usepackage{amsmath,amsthm,amssymb,amsfonts,amscd,epsfig,latexsym,graphicx,textcomp}
\usepackage{tikz-cd}
\usepackage{psfrag}
\usepackage[all]{xy}
\xyoption{pdf}
\usepackage{stackengine}
\usepackage{comment}
\usepackage{hhline}
\usepackage{diagbox}
\usepackage{float}
\usepackage{caption}
\usepackage{eufrak}
\usepackage{slashed}
\usepackage{amsmath, amssymb, graphics, setspace}

\usepackage{listings}

\usepackage{tikz}
\usepackage{stmaryrd}

\restylefloat{figure}

\usepackage{thmtools}
\usepackage{thm-restate}

\usepackage{hyperref}

\usepackage{cleveref}

\newtheorem{theorem}{Theorem}[section]
\newtheorem{lemma}[theorem]{Lemma}

\newtheorem{question}[theorem]{Question}

\theoremstyle{definition}
\newtheorem{definition}[theorem]{Definition} 
\newtheorem{example}[theorem]{Example} 

\theoremstyle{remark}
\newtheorem{remark}[theorem]{Remark}
\numberwithin{equation}{section}

\usepackage{xpatch}
\makeatletter   
\xpatchcmd{\@tocline}
{\hfil\hbox to\@pnumwidth{\@tocpagenum{#7}}\par}
{\ifnum#1<0\hfill\else\dotfill\fi\hbox to\@pnumwidth{\@tocpagenum{#7}}\par}
{}{}
\makeatother    

\makeatletter
 \def\l@subsection{\@tocline{2}{0pt}{4pc}{6pc}{}}
\def\l@subsubsection{\@tocline{3}{0pt}{8pc}{8pc}{}}
 \makeatother

\begin{document}

\title{A counterexample for the Polar Conjecture of Spencer-Brown}

\thanks{}

\author{Scott Baldridge}
\address{Department of Mathematics, Louisiana State University,
Baton Rouge, LA}
\email{baldridge@math.lsu.edu}

\author{Louis H. Kauffman}
\address{Department of Mathematics, Statistics and Computer Science, 851 South Morgan Street, University of Illinois at Chicago,
Chicago, Illinois 60607-7045} 
\address{International Institute for Sustainability with Knotted Chiral Meta Matter (WPI-SKCM2), Hiroshima University, 1-3-1 Kagamiyama, Higashi-Hiroshima, Hiroshima 739-8526, Japan}
\email{kauffman@uic.edu}

\author{Ben McCarty}
\address{Department of Mathematical Sciences, University of Memphis,
Memphis, TN}
\email{ben.mccarty@memphis.edu}

\subjclass{}
\date{}

\begin{abstract} 
In 1976, George Spencer-Brown announced a proof of the four color theorem, using operations on Tait colorings for trivalent plane graphs. In subsequent work he formulated these operations in terms of an algorithm that he called a parity-pass and claimed that when the parity pass algorithm is performed on a non-polar pentagon region, it necessarily terminates in an edge coloring that is extendable to the entire graph.  We provide here a counterexample to show that this claim is false.  We then raise questions related to the existence of this sort of counterexample.
\end{abstract}

\maketitle

\section{Introduction}

First proposed in 1852, the four color theorem stood for about 125 years as a conjecture that inspired the work of many mathematicians.  Its first proof, by Appel and Haken in 1976 \cite{AppelHaken,AppelHaken2} left some in the mathematics community disappointed, due in large part to the massive computer case-checking required.  Around the same time, George Spencer-Brown announced a very different, and crucially, human-checkable, proof of the four color theorem.  It took several years, but in 1980, Spencer-Brown submitted what would become the first account of his method with the Royal Society \cite{CAFPOM}.  This account was later expanded in the mid nineties to become Appendix 5 in his famous book \emph{Laws of Form} \cite{GSB}.

While Spencer-Brown proposed multiple proofs of the four color theorem, one of his main approaches and the one of interest herein is the \emph{parity pass}---a sequence of five modifications one can do repeatedly to a partially edge-colored graph.  He then attempts to show that this sequence produces a coloring that extends to the entire graph.  Spencer-Brown claimed in the early 1980s (cf. Theorem 23 of Appendix 5 in \cite{GSB}) that the parity pass constituted a constructive proof of the four color theorem.  However, around that time he became aware that his algorithm could in fact result in a non-terminating algorithm on the famous Errera map (cf. \cite{Errera,Kittell}), when performed at one of two polar pentagons in the map.  By polar, he meant a region around which the graph is radially symmetric.  At any other pentagon, he observed that the parity pass was successful in producing a $3$-edge coloring.  Spencer-Brown claimed that the non-terminating algorithm could only occur on a polar pentagon.  He argued that, aside from a few exceptional cases he could list, every map would have a non-polar pentagon.  This lead to Theorem 25 of Appendix 5 in \cite{GSB}, where he claims that as long as the pentagon at which the parity pass algorithm is performed is non-polar, the algorithm will succeed in finding an extendable coloring.  He did not, however, supply a proof of this theorem.  Therefore we refer to it as the \emph{Polar Conjecture}.  It is this conjecture for which we present a counterexample.  

Dating the conjecture precisely is difficult.  The second author was aware of it by 1991, and it seems likely that it originated sometime in the 1980s via notes that were circulated among various researchers, including the second author.  This makes the conjecture nearly 40 years old.  

Finding the counterexample was a nontrivial exercise in experimentation guided by intuition and the tools developed in \cite{BM-Color,BM-Reduce}, making it unsurprising that it took almost 40 years to find a counterexample.  For example, our main counterexample requires sixty steps for verification (cf. \Cref{sec:CX}).

The paper is organized as follows.  In \Cref{sec:ToEdges}, we discuss formations and how they relate to the edge coloring conventions used in this paper.  In \Cref{sec:ParityPass}, we discuss the parity pass algorithm.  \Cref{sec:CX} introduces the counterexample.  Finally, in \Cref{sec:Implications}, we discuss the implications of this counterexample for Spencer-Brown's proposed proof.

\section{Formations, Face Colorings, and Edge Colorings}\label{sec:ToEdges}
At the heart of Spencer-Brown's work on the four color theorem is the concept of a formation, which we define below.

\begin{definition}[cf. \cite{StateCalc,MapColoring,MapReform,ReformMap,GSB}]
A \emph{formation} is a finite collection of simple closed curves in the plane, each colored red or blue, satisfying:  
\begin{enumerate}
\item Curves of the same color are pairwise disjoint.
\item A red curve may intersect a blue curve along a finite number of arcs, where the intersections may involve a crossing or not.
\end{enumerate}
\end{definition}

Associated to a formation $F$, there is a well-defined trivalent graph $G(F)$ obtained by identifying the arcs where red and blue curves meet with edges, and the endpoints of those edges with vertices.  The remaining red and blue arcs, where no interaction occurs, form the remaining edges of $G(F)$.  

Given a trivalent graph $G(V,E)$ and a formation $F$ such that $G = G(F)$, a 3-face colored ribbon graph is constructed as follows.  At each intersection of red and blue curves, insert a twisted band if the curves cross, or an untwisted band otherwise. For edges corresponding to a single red or blue arc, replace the arc with a ribbon by adding a green arc on the opposite side, incorporating twists as necessary to maintain consistency. This construction is illustrated in \Cref{fig:FormToRG} and further described in \cite{BM-Reduce}.  Let such a ribbon graph be called a \emph{state graph} with a 3-face coloring (cf. \cite{BM-Color}).

\begin{figure}[h]
\includegraphics[scale=0.5]{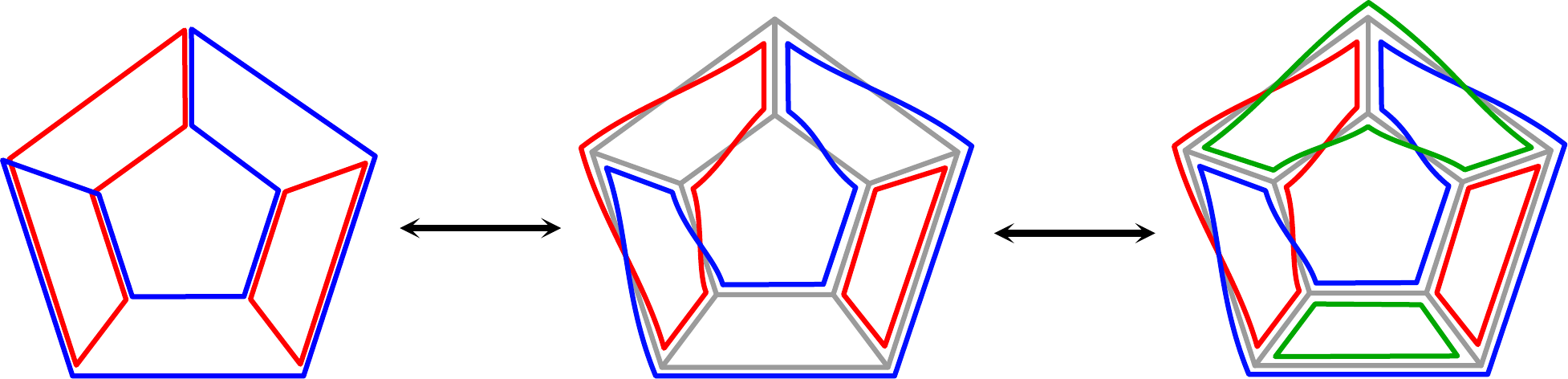}
\caption{The relationship between formations and $3$-face colorings of a ribbon graph.  The underlying trivalent graph, is shown in grey.}
\label{fig:FormToRG}
\end{figure}

Given a $3$-face coloring of a state graph, one can construct an associated $3$-edge coloring as follows.  For an edge of the graph incident to faces colored $c_1$ and $c_2$ (with $c_1 \neq c_2$), assign the remaining color $c_3$ to the edge.  This assignment corresponds to the group operation in the Klein $4$-group, where each color is assigned to one of the three nontrivial elements.  For the $3$-face coloring in \Cref{fig:FormToRG}, the corresponding $3$-edge coloring is given in \Cref{fig:ThreeEdge}.

\begin{figure}[H]
\includegraphics[scale=0.5]{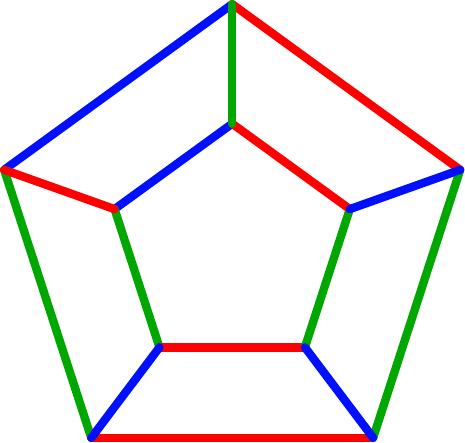}
\caption{The $3$-edge coloring associated to the $3$-face coloring in \Cref{fig:FormToRG}}
\label{fig:ThreeEdge}
\end{figure}

There is another natural way to construct an edge coloring of the graph from a formation.  Edges corresponding to intersections of red and blue curves are colored green, while edges corresponding to a single red or blue arc inherit that arc's color. This coloring corresponds to the 3-edge coloring derived from the 3-face coloring, up to a permutation of red and blue colors across the graph.

In Spencer-Brown's work~\cite{GSB}, formations underpin his algorithmic approach to the four-color theorem. For our purposes, however, working with 3-edge colorings simplifies connections to 3-face colorings of ribbon graphs, and allows us to describe Spencer-Brown's algorithmic approach to the four color theorem in the language of Tait switches. Accordingly, we adopt 3-edge colorings as our primary framework and, in the next section, present an equivalent formulation of Spencer-Brown's algorithm in these terms.

\section{The Parity Pass}\label{sec:ParityPass}

In a potential minimal counterexample to the four-color theorem, Kempe's results imply that there must be a face with at most five sides.  While Kempe's methods successfully handle faces with fewer than five sides, the pentagon requires a new approach. To that end, Spencer-Brown introduces the \emph{parity pass} algorithm, which operates on a partial 3-edge coloring of the graph where all but one edge of the pentagon is colored, termed a \emph{1-deficient} coloring. The algorithm proceeds by applying a sequence of Kempe-like switches, each involving an exchange of two alternating colors along a path outside the local pentagon region, where edges of the pentagon may be colored or uncolored as a result of the process.

The parity pass algorithm is illustrated in \Cref{fig:ParityPass} (cf. \cite{ReformMap}). The first figure in the top left depicts the 1-deficient coloring around the pentagon, with the uncolored edge indicated by a dotted line. The rest of the graph is assumed to have a 3-edge coloring consistent with this local configuration.

\remark Suppose one could find a minimal counterexample to the four-color theorem with a pentagon face. Removing one edge of the pentagon, as depicted in the top left of \Cref{fig:ParityPass}, yields a graph that must admit a 3-edge coloring. One might question why the local coloring around the pentagon must match the configuration shown in \Cref{fig:ParityPass}. In truth, it need not be identical.  However, if the local coloring differs from this (up to color permutation), it can be shown that the coloring can be extended to the entire graph, possibly after doing one or two simple two-color switches outside the pentagon. We omit the proof here but assume this local coloring as the starting point, as it represents, up to color permutation, the starting point for Spencer-Brown's parity pass algorithm (cf.~\cite{ReformMap,GSB}).

\begin{figure}[H]
\includegraphics[scale=0.7]{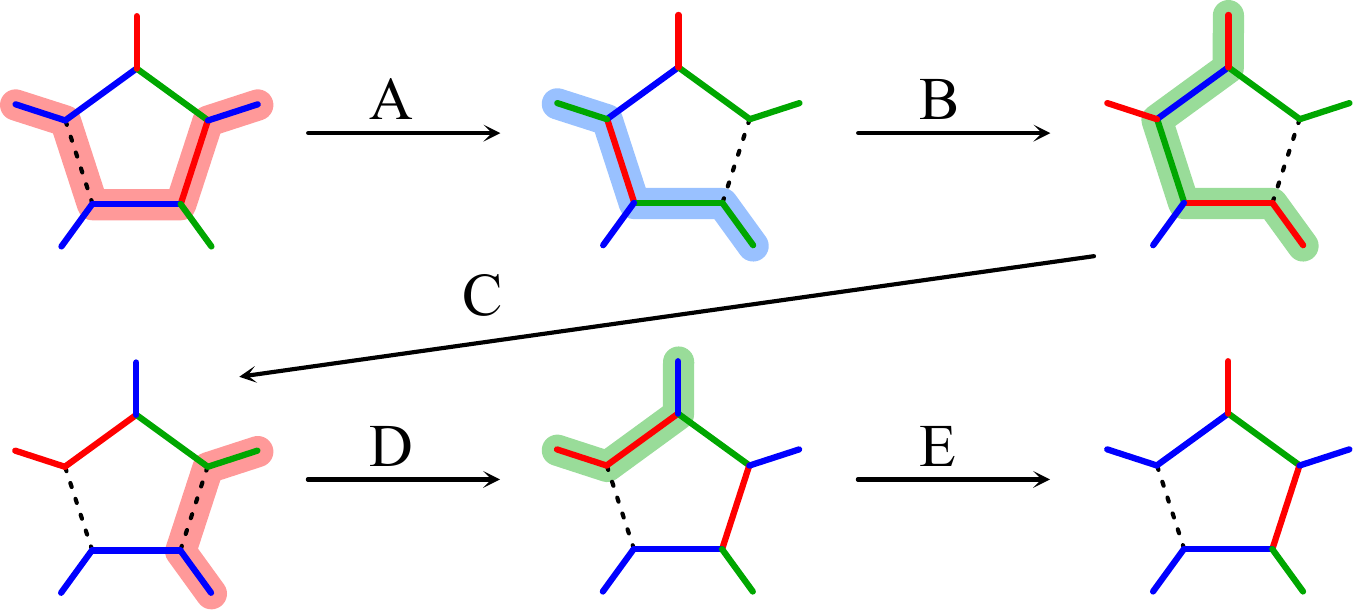}
\caption{One complete loop of the parity pass algorithm.}
\label{fig:ParityPass}
\end{figure}

In each diagram of \Cref{fig:ParityPass}, a highlighted edge path indicates the type of alternating two-color path followed outside the local pentagon region, using the third color for the highlight. Specifically, a red highlight corresponds to a blue-green path, a blue highlight to a red-green path, and a green highlight to a red-blue path. The algorithm assumes that, at each step, the partially colored graph outside the local region contains the necessary alternating edge-color path from one endpoint of the highlighted path to the other. If such a path does not exist, Spencer-Brown asserts that the coloring at the current step can be completed using simple operations, like a Tait switch as in step B of \Cref{fig:ParityPass}, to produce a 3-edge coloring that directly extends to the entire graph. His assertion is true, but we do not prove it here, as it is not required for the analysis of our counterexample.  Certain steps in the algorithm switch the parity of the number of alternating color paths, and others do not.  See \cite{ReformMap,CAFPOM} for details.

After performing the five switches (steps A--E) in \Cref{fig:ParityPass}, the local edge coloring around the pentagon returns to its initial configuration. However, the edge coloring outside this local region may differ significantly from the starting coloring. If the 3-edge coloring cannot be extended to the entire graph at this point, the algorithm restarts from step A. In some cases, the algorithm terminates quickly, yielding a coloring extendable to the entire graph in only a few steps. In others, multiple iterations of the five switches may be required to produce a coloring that extends to the pentagon (cf.~\cite{ReformMap,GSB}; see also \Cref{ex:LoopThenStop}).

A key question is whether there exists a graph with a $1$-deficient 3-edge coloring where the parity pass algorithm enters a non-terminating loop. In such a scenario, at each step, the coloring cannot be completed by a simple operation, and the required alternating edge path exists to proceed to the next step. Since there are only a finite number of possible $1$-deficient $3$-edge colorings, such an example would, after some number of iterations return to the original coloring, not just locally, but over the entire graph.

\begin{example}
\Cref{fig:Errera} shows a $1$-deficient coloring of the Errera graph where the parity pass algorithm applied at the central pentagon, does not terminate.  We leave this as an exercise for the reader.  This example was known to Spencer-Brown in the 1980s and led to the Polar Conjecture, as the central pentagon is indeed polar. 
\begin{figure}[H]
\includegraphics[scale=0.5]{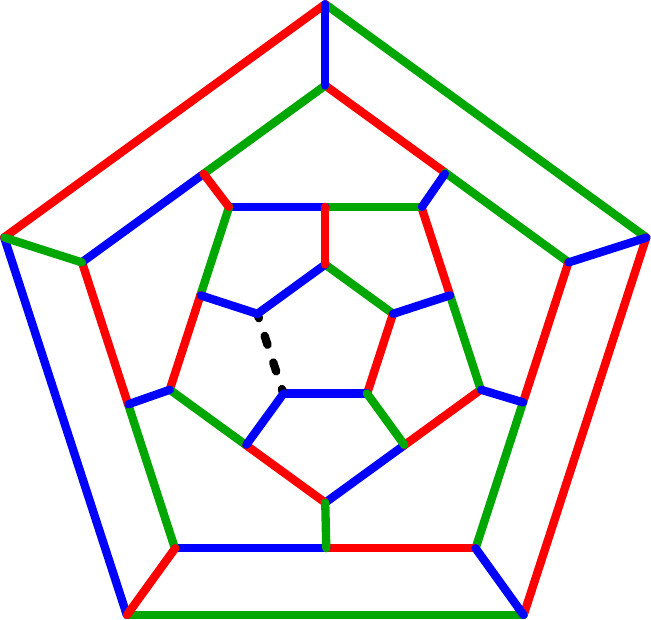}
\caption{A $1$-deficient coloring of the Errera graph, which exhibits non-terminating behavior under the parity pass algorithm.}
\label{fig:Errera}
\end{figure}

\end{example}

We next introduce a configuration that will be of use in analyzing the graph at each stage of the parity pass algorithm.


\subsection{A bad configuration}
\Cref{fig:BadCap} depicts a configuration in a planar graph that prevents a partial $3$-edge coloring from being extended to a pentagon through simple color switches on paths that begin and end at the pentagon.  Throughout this section, we assume $\Gamma$ is a plane graph which contains a pentagonal face, with the edges incident to that pentagonal face termed \emph{spokes}.  The highlighted paths in \Cref{fig:BadCap} are intentionally ambiguous to reflect their variable routes through the graph across different examples.  The important features to recognize are that the alternating color paths begin and end at the spokes as shown, and that the two alternating color paths (the red and green highlights) interact---meaning that they overlap in at least one edge (for the pictured example, they would interact along blue edges).  This is, of course, the edge coloring version of the famous Kempe chain problem.

\begin{figure}[H]
\includegraphics[scale=0.7]{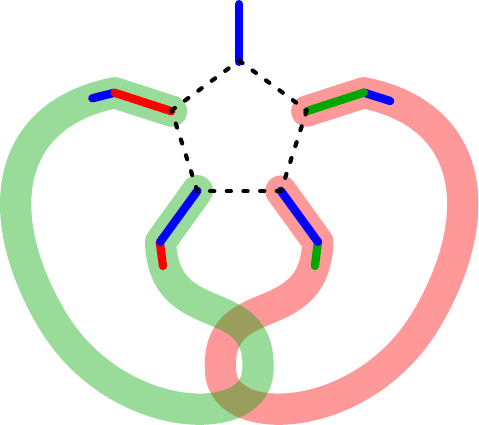}
\caption{A bad configuration, $\mathcal{C}$.  }
\label{fig:BadCap}
\end{figure}

\begin{theorem}\label{thm:NoSimpleOp}
Let $\Gamma$ be a plane graph with a pentagonal face and a $1$-deficient $3$-edge coloring.  Suppose that the coloring outside of the pentagon matches configuration $\mathcal{C}$ (see \Cref{fig:BadCap}) up to color permutation or basic symmetry.  Then the coloring cannot be extended to the pentagon, even after applying a simple two color switch that begins at the pentagon.
\end{theorem}

Before proving the theorem we prove the following lemma.

\begin{lemma}\label{lem:NoExtend}
Let $\Gamma$ be a plane graph with a pentagonal face.  Suppose that the edges of the graph are properly colored with three colors, except at the boundary of the chosen pentagon, which are left uncolored.  Suppose further that the spokes are colored as in \Cref{fig:BadCap}.  Then the coloring cannot be directly extended to the pentagon.
\end{lemma}

\begin{proof}
Consider the configuration $\mathcal{C}$.  Between a blue and green spoke there must be a red edge in the pentagon, and between a red and blue spoke there must be a a green edge.  Attempting to color in this way necessarily leads to a contradiction.
\end{proof}

We now prove \Cref{thm:NoSimpleOp}.

\begin{proof}
Given the coloring does not extend directly, by \Cref{lem:NoExtend}, there are only five possible two color switches to check that begin and end at spokes.  Each one results in a coloring of the spokes that is either equivalent to the one in \Cref{fig:BadCap} or is a rotation of it.  By symmetry, and \Cref{lem:NoExtend}, the coloring cannot extend to the pentagon.
\end{proof}

\begin{example}\label{ex:DoubleSwitch}
If one could do a color switch along both the red and green highlighted paths in \Cref{fig:BadCap}, the resulting $3$-edge coloring would extend to the pentagon as shown in \Cref{fig:DoubleSwitch}.  

\begin{figure}[H]
\includegraphics[scale=0.7]{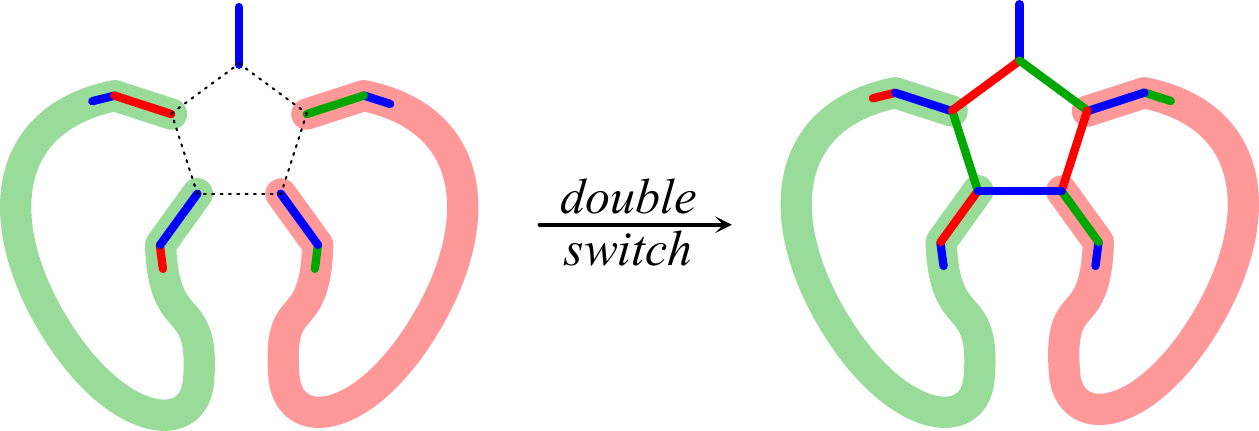}
\caption{A double switch allows the coloring to extend.}
\label{fig:DoubleSwitch}
\end{figure}

However, in configuration $\mathcal{C}$ the red and green highlighted paths, representing alternating two-color paths, share at least one blue edge, as indicated by the crossing in \Cref{fig:BadCap}.  Consequently, the double switch depicted in \Cref{fig:DoubleSwitch} may not be possible for a partially colored graph that contains $\mathcal{C}$.

\end{example}

Readers familiar with the four-color theorem will recognize configuration $\mathcal{C}$ as an instance of the classical Kempe chain problem, where two alternating two-color paths, intended for color switches, share at least one edge, preventing both switches from being performed simultaneously.

\Cref{thm:NoSimpleOp} is stronger than needed for the parity pass algorithm, asserting that even a non-simple switch (cf. ``simple operation'' in \cite{ReformMap,GSB}) is insufficient for obtaining a valid coloring of the graph, when configuration $\mathcal{C}$ is present in the graph.  Nevertheless this theorem provides a practical test at each step of the parity pass algorithm:  if configuration $\mathcal{C}$ is detected, the algorithm continues.  The next theorem says that at each stage of the algorithm, if configuration $\mathcal{C}$ is detected (possibly rotated), then the required alternating two-color path for the next step of the algorithm exists in the graph.

\begin{theorem}\label{thm:BadConfigPass}
For any step of the parity pass algorithm (see Figure~\ref{fig:ParityPass}), configuration $\mathcal{C}$, up to color permutation or rotation, is compatible with the local coloring of the pentagonal region. Moreover, when $\mathcal{C}$ is present, the graph contains the alternating two-color path required to perform the next step of the algorithm.
\end{theorem}

\begin{proof}
We demonstrate the argument for the configuration in the domain of step C of the parity pass algorithm (see \Cref{fig:ParityPass}), with other steps following similarly.  As shown in \Cref{fig:BadCapOnPent}, there must be a red-blue (highlighted green) path from the top-most spoke to the lower-right spoke.  The existence of this path is required since we know that there is already a separate red-blue path between the two lefthand spokes.  This results in a loop that allows step C to perform a red-blue switch along the arc, while also uncoloring the green edge.
\begin{figure}[h]
\includegraphics[scale=0.7]{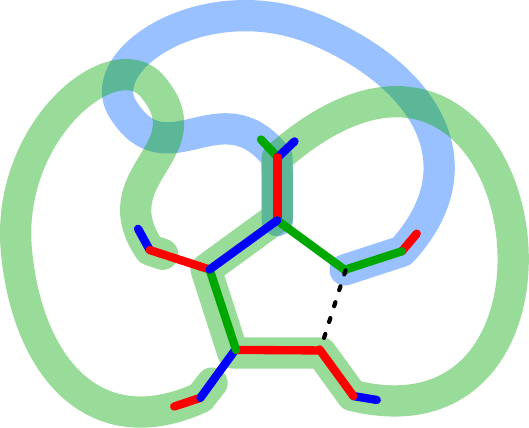}
\caption{Configuration $\mathcal{C}$ on the domain of step C of the parity pass.}
\label{fig:BadCapOnPent}
\end{figure}
\end{proof}

\Cref{thm:BadConfigPass} enables the following description of the parity pass algorithm.  Given a plane graph with a pentagonal face and a $1$-deficient $3$-edge coloring, if the bad configuration $\mathcal{C}$ is present, perform transformation A in \Cref{fig:ParityPass}.  If the bad configuration $\mathcal{C}$ persists, perform the next transformation in sequence.  Continue until the bad configuration is absent, or until the original $1$-deficient $3$-edge coloring (for the entire graph) is reached.  For our counterexample, this latter state is what we will see. 

\section{The Counterexample}\label{sec:CX}
In \Cref{fig:CX} we present the underlying graph $\mathcal{G}$ used to construct our main counterexample to the Polar Conjecture.  A counterexample to the Polar Conjecture requires a graph lacking radial symmetry around the pentagon where the parity pass algorithm is applied.  The chosen pentagon is highlighted in yellow.

The graph lacks radial symmetry about the highlighted pentagon as follows. Consider $\mathcal{G}$ embedded on the 2-sphere, with the pentagon at one pole. The graph’s single highlighted octagonal face must then lie at the opposite pole. Any rotational symmetry fixing both the pentagon and the octagon must have an order dividing both 5 and 8, which have no common nontrivial divisors. Thus, no nontrivial rotational symmetry exists.

\begin{figure}[H]
\includegraphics[scale=0.5]{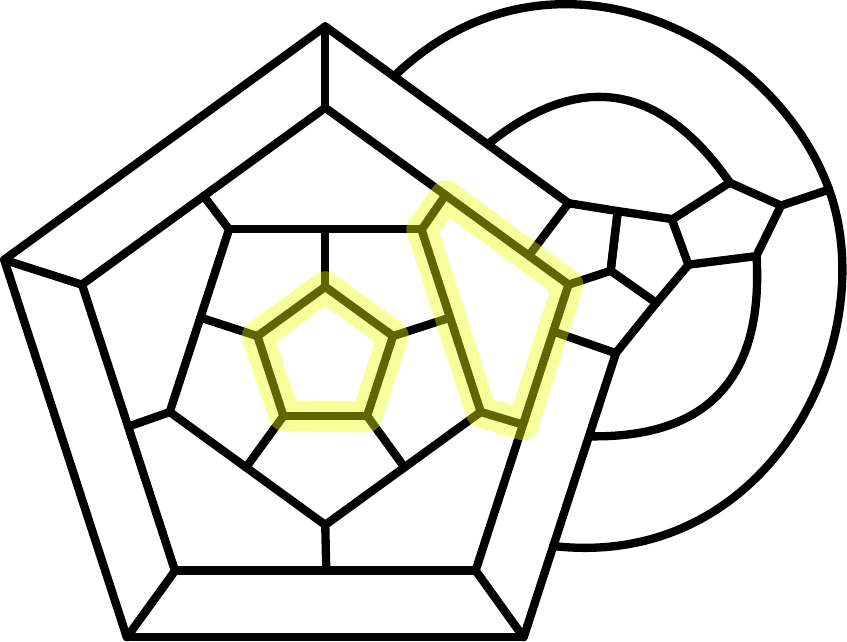}
\caption{The graph $\mathcal{G}$ for our counterexample.}
\label{fig:CX}
\end{figure}

\remark The graph $\mathcal{G}$ exhibits reflection symmetry, but the Polar Conjecture requires rotational symmetry, not merely reflective symmetry.  For a counterexample that lacks reflective symmetry as well, see \Cref{ex:NoReflection}.

\Cref{fig:CXFirstPass} shows one complete parity pass applied to a $1$-deficient coloring of $\mathcal{G}$.  Observe that the first partial coloring in the sequence exhibits the bad configuration, $\mathcal{C}$, and in each subsequent step it persists (see \Cref{fig:BadCapInPass} for the first and last steps).  

The coloring after step E is locally identical to the initial coloring, however, the colors outside the pentagon and its spokes are quite different.  Since the parity pass algorithm operates locally, it can be reapplied.  The reader is encouraged to try this and observe that the bad configuration appears in each step through twelve complete cycles, or sixty steps, before finally returning to the original edge coloring over the entire graph.  Thus, this example invalidates the Polar Conjecture.  

\begin{figure}[h]
\includegraphics[scale=0.45]{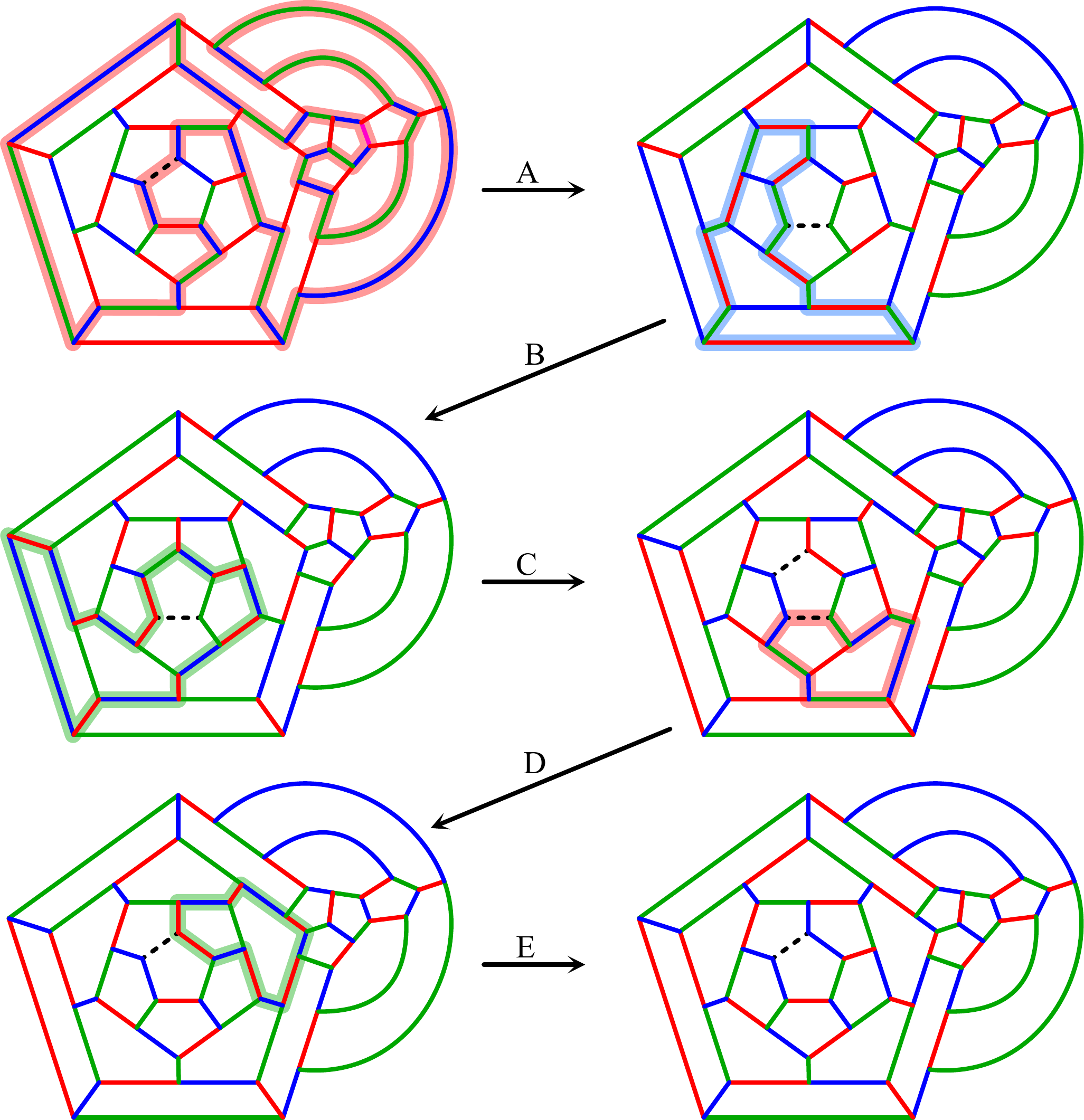}
\caption{One complete parity pass on the counterexample.}
\label{fig:CXFirstPass}
\end{figure}

\begin{figure}[h]
\includegraphics[scale=0.45]{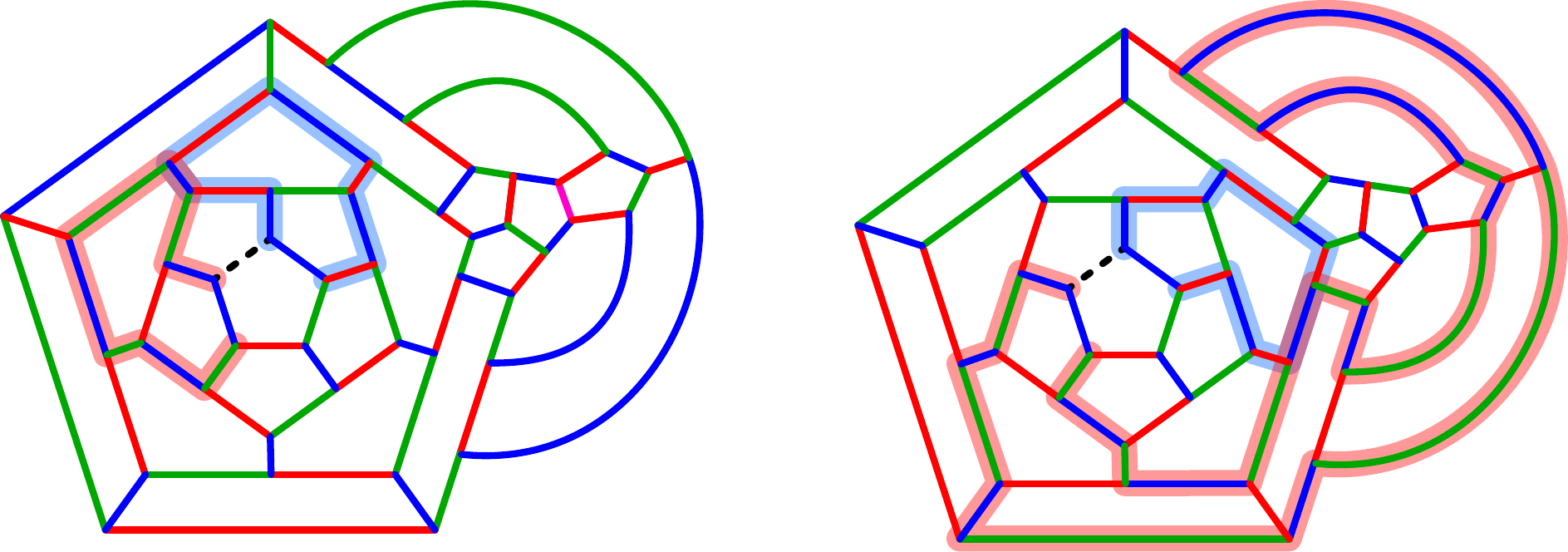}
\caption{The configuration $\mathcal{C}$ before and after one complete parity pass.}
\label{fig:BadCapInPass}
\end{figure}

\begin{example}\label{ex:NoReflection}
\Cref{fig:NoReflection} depicts a $1$-deficient coloring of a graph derived from our counterexample $\mathcal{G}$ by doing an I-H move on the yellow highlighted edge.  This move introduces a four-sided face which can be addressed using standard Tait or Kempe switch methods.  
However, the graph retains the bad configuration $\mathcal{C}$ (see \Cref{fig:BadCap}). When the parity pass algorithm is applied, it exhibits looping behavior, similar to our main counterexample $\mathcal{G}$, completing four full cyles (20 steps) before returning to the original $1$-deficient coloring. 

\begin{figure}[h]
\includegraphics[scale=0.5]{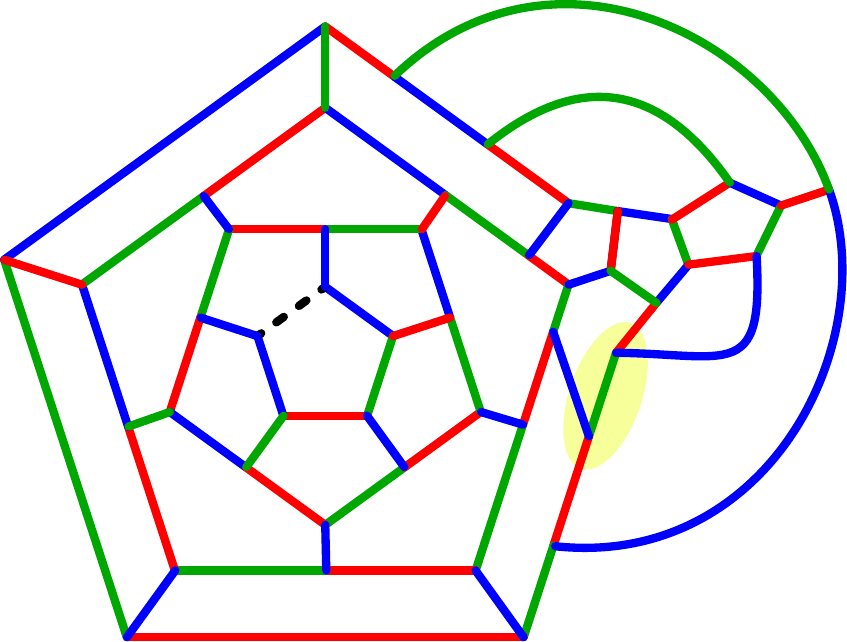}
\caption{A counterexample without a line of symmetry.}
\label{fig:NoReflection}
\end{figure}

\end{example}

\begin{example}\label{ex:LoopThenStop}
\Cref{fig:FiniteLoopEx} depicts a $1$-deficient coloring of a graph derived from the graph in \Cref{ex:NoReflection} by inserting an edge to split the lower-rightmost face (excluding the outer region) and adjusting the edge coloring accordingly.  Two observations follow.  First, the bad configuration $\mathcal{C}$ is present in this coloring.  Second, applying the parity pass algorithm, the first two complete cycles (ten steps) retain some form of configuration $\mathcal{C}$.  However, after the first step of the third cycle (the A step), $\mathcal{C}$ is absent, and the coloring is extendable (i.e. completable by simple operations in the language of \cite{ReformMap,GSB}).  

\begin{figure}[h]
\includegraphics[scale=0.5]{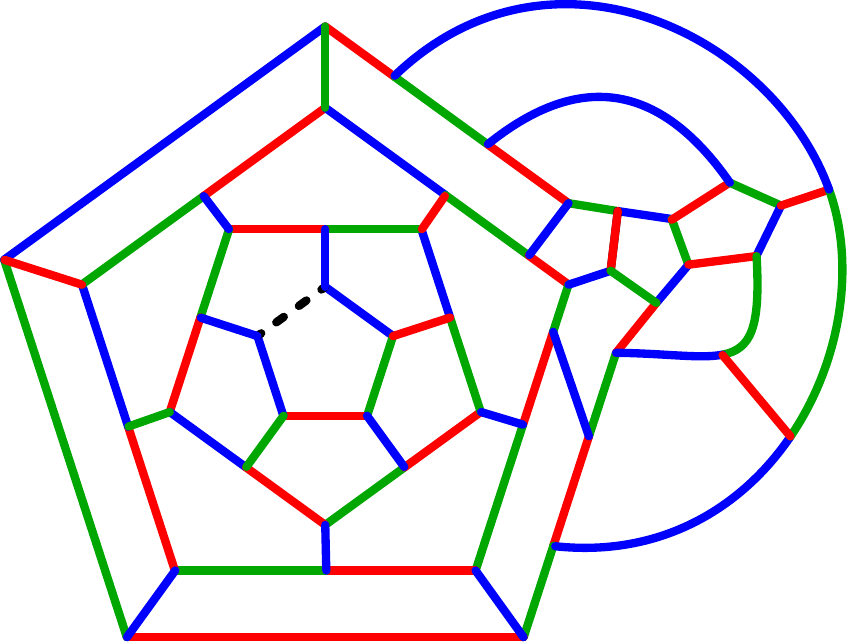}
\caption{An example where the parity pass loops, but terminates.}
\label{fig:FiniteLoopEx}
\end{figure}

This example illustrates the subtlety of the looping behavior in the parity pass algorithm. To detect a non-terminating loop, it is necessary to continue the algorithm until the initial $1$-deficient $3$-edge coloring of the entire graph is repeated while noting that the bad configuration $\mathcal{C}$ persists in each step.
\end{example}

\section{Implications for Spencer-Brown's proposed proof}\label{sec:Implications}
Our counterexample demonstrates that the Polar Conjecture is false for the parity pass algorithm used in the proof of Theorem 23 in Appendix 5 of \cite{GSB}.  However, the counterexample does not disprove Spencer-Brown's broader claim that, in any potential counterexample to the four-color theorem, there exists some pentagonal face where the parity pass algorithm produces a valid 3-edge coloring.  Unfortunately, if the algorithm does not work for an arbitrary pentagon, or even an arbitrary non-polar pentagon, then one is propelled into a search across all pentagons in the map---thus, killing the effectiveness of such an algorithm.

Theorem 24 in Appendix 5 of~\cite{GSB} asserts (without proof) that, for any non-polar pentagonal face, there exists a sequence of ``color exchanges'' that yields a 3-edge coloring extending to the pentagon. This allows greater flexibility than the fixed sequence of transformations in the parity pass algorithm. In our main counterexample graph $\mathcal{G}$, the initial 1-deficient 3-edge coloring is, in fact, completable via a sequence of simple operations, as is the Errera graph, which we leave as an exercise for the reader.  This broader claim is not addressed here, and remains open.  

The ability to construct non-symmetric graphs where configuration $\mathcal{C}$ (see \Cref{fig:BadCap}) persists through Spencer-Brown’s sequence of transformations might raise reasonable doubts about whether the four-color theorem can be proven using the pentagonal face, as in Kempe’s approach. However, several open questions merit exploration.

\begin{question}\label{ques:AllPentagonsLoop}
Do there exist non-symmetric plane graphs with non-polar pentagons, such that every pentagonal face admits a 1-deficient 3-edge coloring that exhibits looping behavior under the parity pass algorithm?
\end{question}

An even stronger counterexample to the polar conjecture would require the above property for all pentagonal faces; otherwise, the parity pass applied to a pentagon without looping would produce a valid coloring. Constructing such a graph appears challenging, suggesting an intermediate question:

\begin{question}\label{ques:TwoPentagonsLoop}
Do there exist non-symmetric plane graphs with at least two non-polar pentagonal faces, each admitting a 1-deficient 3-edge coloring that exhibits looping behavior under the parity pass algorithm?
\end{question}

Finally, we note that Spencer-Brown gave other algorithms as well that he described as ``parity mills."    Elsewhere in a modified version of \cite{CAFPOM} that circulated in the 1980s, and also described in the 1990s in \cite{GSB2,GSB3}, he gives two parity mills that he calls $\phi_1$ and $\phi_2$ which are pictured in \Cref{fig:AltPass1,fig:AltPass2}.  In testing these two parity mills on our counterexample we used the first three transformations of the parity pass to obtain a $2$-deficient coloring which is required for the beginning of the parity mills (see the coloring after transformation C in \Cref{fig:CXFirstPass}).  Then, applying the parity mills, we found that $\phi_1$ terminates in a $2$-deficient coloring that can be converted to a valid coloring by the double switch of \Cref{fig:DoubleSwitch}.  However, $\phi_2$ is non-terminating, just like the parity pass.

\begin{figure}[h]
\includegraphics[scale=0.7]{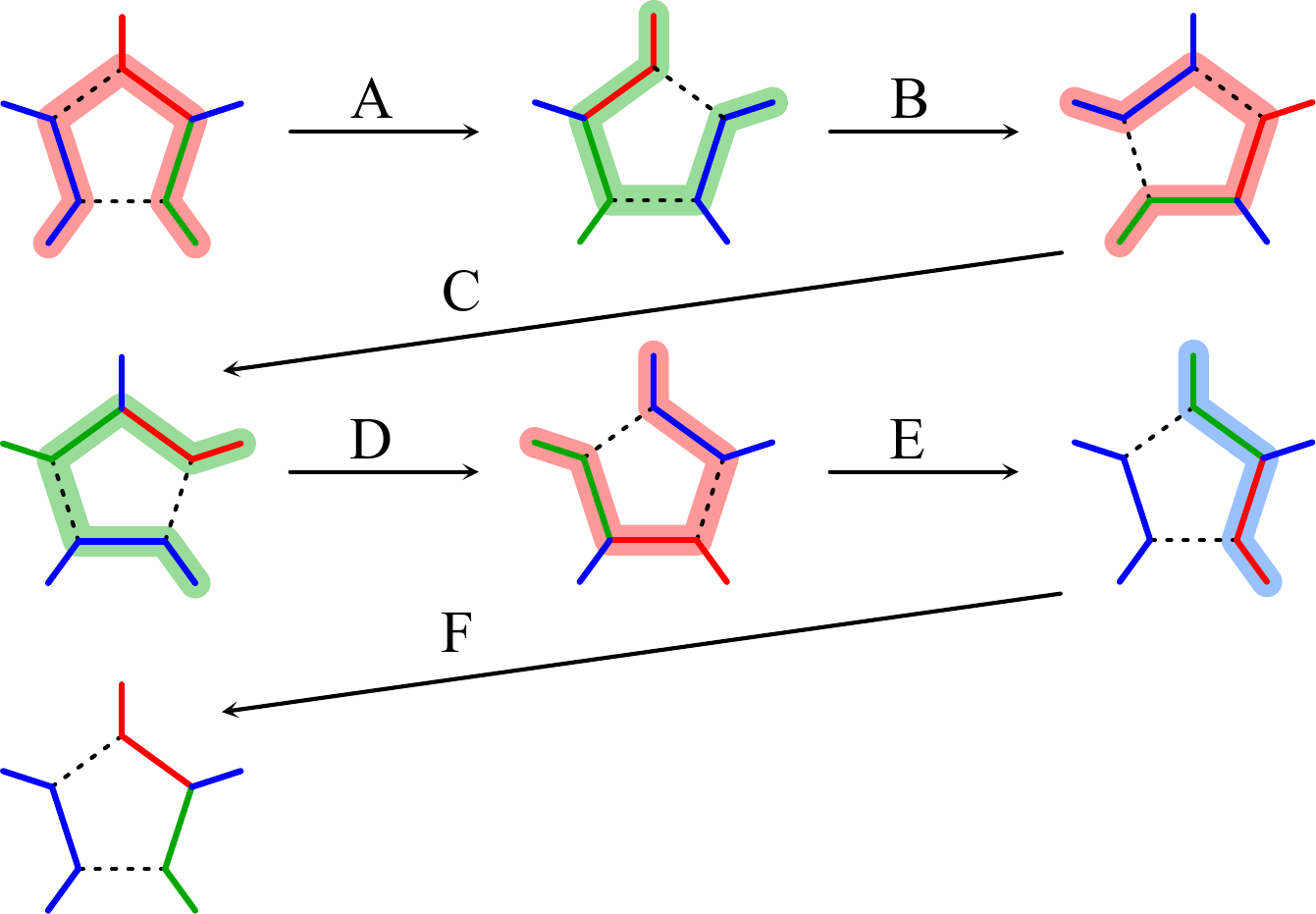}
\caption{The parity mill $\phi_1$.}
\label{fig:AltPass1}
\end{figure}

\begin{figure}[h]
\includegraphics[scale=0.7]{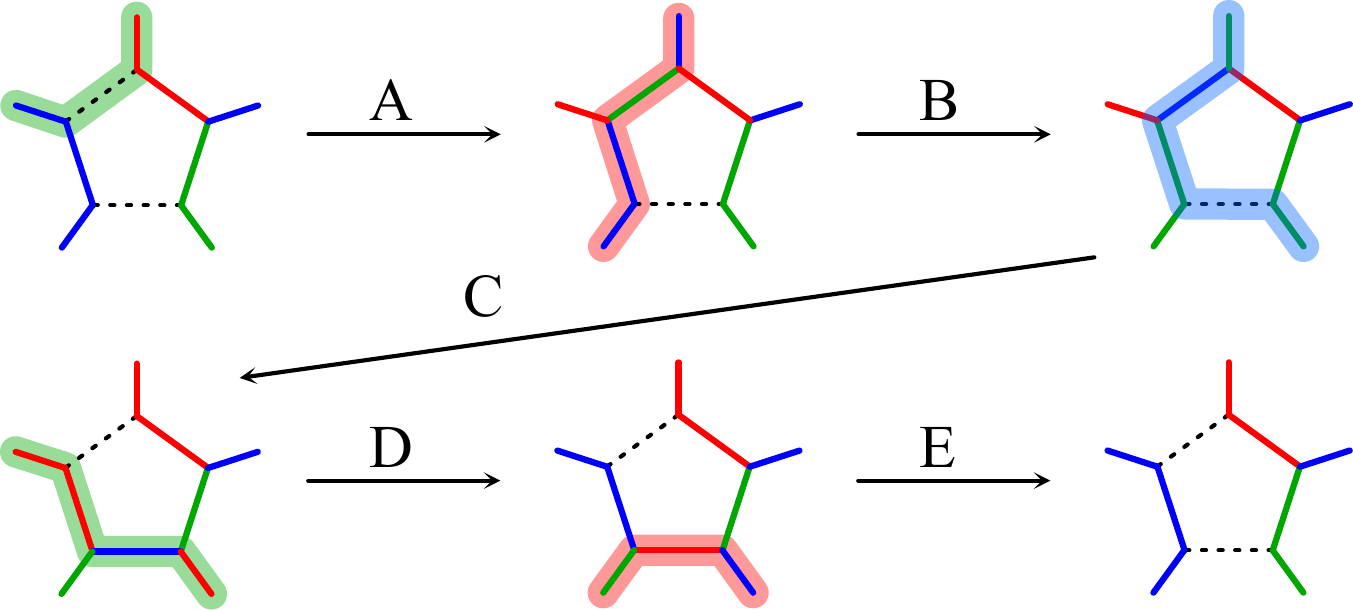}
\caption{The parity mill $\phi_2$.}
\label{fig:AltPass2}
\end{figure}

Suppose that one is equipped with the parity pass in \Cref{fig:ParityPass} as well as the two parity mills $\phi_1$ and $\phi_2$ shown in \Cref{fig:AltPass1,fig:AltPass2}. As previously noted, our counterexample works for two of the  algorithms but not the third. One might naturally ask the following question, which is currently open.

\begin{question}\label{ques:AllPassesLoop}
Do there exist non-symmetric plane graphs with a non-polar pentagon for which all three algorithms, the parity pass, and the two parity mills, exhibit looping behavior?
\end{question}

These algorithms are based on moves from the impasse group of \cite{Kittell}.  Thus there is some hope that the impasse group provides sufficient flexibility to find a coloring.

\end{document}